\documentclass [12pt,a4paper]{amsart}

\usepackage{amssymb,amsmath,amsthm}

\usepackage{url}

\begin{document}

\newtheorem{theorem}{Theorem}[section]
\newtheorem{lemma}[theorem]{Lemma}
\newtheorem{proposition}[theorem]{Proposition}

\newtheorem{statement}[theorem]{Statement}
\newtheorem{conjecture}[theorem]{Conjecture}
\newtheorem{problem}[theorem]{Problem}
\newtheorem{corollary}[theorem]{Corollary}

\newtheorem*{maintheorem1}{Theorem A}
\newtheorem*{maintheorem2}{Theorem B}

\newtheoremstyle{neosn}{0.5\topsep}{0.5\topsep}{\rm}{}{\bf}{.}{ }{\thmname{#1}\thmnumber{ #2}\thmnote{ {\mdseries#3}}}
\theoremstyle{neosn}
\newtheorem{remark}[theorem]{Remark}
\newtheorem{definition}[theorem]{Definition}
\newtheorem{example}[theorem]{Example}

\numberwithin{equation}{section}

\newcommand{\Aut}{\,\mathrm{Aut}\,}
\newcommand{\Inn}{\,\mathrm{Inn}\,}
\newcommand{\End}{\,\mathrm{End}\,}
\newcommand{\Out}{\,\mathrm{Out}\,}
\newcommand{\Hom}{\,\mathrm{Hom}\,}
\newcommand{\ad}{\,\mathrm{ad}\,}
\newcommand{\SLn}{$\rm sl(n)$}
\newcommand{\calL}{{\mathcal L}}
\renewenvironment{proof}{\noindent \textbf{Proof.}}{$\blacksquare$}
\newcommand{\GL}{\,\mathrm{GL}\,}
\newcommand{\SL}{\,\mathrm{SL}\,}
\newcommand{\PGL}{\,\mathrm{PGL}\,}
\newcommand{\PSL}{\,\mathrm{PSL}\,}
\newcommand{\PSO}{\,\mathrm{PSO}\,}
\newcommand{\SO}{\,\mathrm{SO}\,}
\newcommand{\Sp}{\,\mathrm{Sp}\,}
\newcommand{\PSp}{\,\mathrm{PSp}\,}
\newcommand{\UT}{\,\mathrm{UT}\,}
\newcommand{\Exp}{\,\mathrm{Exp}\,}
\newcommand{\Rad}{\,\mathrm{Rad}\,}
\newcommand{\charr}{\mathrm{char}\,}
\newcommand{\tr}{\mathrm{tr}\,}
\newcommand{\diag}{{\rm diag}}

\font\cyr=wncyr10 scaled \magstep1%

\def\Sha{\text{\cyr Sh}}

\title{Isotypical equivalence of periodic Abelian groups}


\author{Elena Bunina}
\date{}
\address{Department of Mathematics, Bar--Ilan University, 5290002 Ramat Gan, ISRAEL}
\email{helenbunina@gmail.com}

\begin{abstract}
	In this paper we give invariants that characterize isotypically equivalent Abelian periodic groups. Also, we describe types of standart tuples of elements in these groups.
	As the particular case we prove that two  Abelian $p$-groups with separable reduced parts are isotypically equivalent if and only if their divisible parts and their basic subgroups are elementarily equivalent. Also as a corollary we prove that any Abelian $p$-group with a separable reduced part is $\omega$-strongly homogeneous.
\end{abstract}

\maketitle

\section{Introduction}\leavevmode

In this paper we study  periodic Abelian groups~$A$. Our goal is to figure out how close are such groups in  case they have the same sets of types realized in~$A$. 
We show that these groups have the same types (are \emph{isotypically equivalent}) if and only if their special numerical invariants coincide. Also, we describe types of 
tuples of standart (in some sense) elements of periodic Abelian groups.

In the particular case we prove that two  Abelian $p$-groups with separable reduced parts are isotypically equivalent if and only if their divisible parts and their basic subgroups are elementarily equivalent. Also as a corollary we prove that any Abelian $p$-group with a separable reduced part is $\omega$-strongly homogeneous.

In this paper (Abelian) groups are our main subject, so we do not consider here rings,
semigroups, etc., even thought most of definitions below make sense for arbitrary algebraic structures.

\begin{definition}
Let $G$ be a group and $(g_1, \dots , g_n)$ a tuple of its elements. The {\bf type} of this
tuple in~$G$, denoted $\mathrm{tp}^G(g_1, \dots , g_n)$, is the set of all first order formulas in free
variables $x_1, \dots , x_n$ in the standard group theory language  which are
true on $(g_1, \dots , g_n)$ in~$G$ (see~\cite{Mya2} or \cite{Mya3} for details). 
\end{definition}

\begin{definition}
The set of all types of tuples
of elements of~$G$ is denoted by $\mathrm{tp}(G)$. Following~\cite{Mya7}, we say that two groups
$G$ and $H$ are {\bf isotypic} if $\mathrm{tp}(G) = \mathrm{tp}(H)$, i.\,e., if any type realized in~$G$ is realized in~$H$, and vice versa.
\end{definition}

 Isotypic groups appear naturally in
logical (algebraic) geometry over groups which was developed in the works of
B.\,I.\,Plotkin and his co-authors (see~\cite{Mya5,Mya6,Mya7} for details), they play an important
role in this subject. In particular, it turns out that two groups are logically equivalent if and only if they have the same sets of realizable types. So there arise  two fundamental algebraic
questions which are interesting in its own right: what are possible types of
elements in a given group~$G$ and how much of the algebraic structure of~$G$ is
determined by the types of its elements?

Isotypicity property of groups is related to the elementary equivalence property,
though it is stronger. Indeed, two isotypic groups are elementarily equivalent,
but the converse does not hold. For example, if we denote by $F_n$ a free
group of finite rank~$n$, then groups $F_n$ and $F_m$ for $2\leqslant m < n$ are elementarily
equivalent~\cite{Mya22,Mya23}, but not isotypic~\cite{Myas}. Furthermore, Theorem~1 from~\cite{Myas}
 shows that two finitely generated isotypic nilpotent groups are isomorphic,
but there are examples, due to Zilber, of two elementarily equivalent non-isomorphic
finitely generated nilpotent of class~$2$ groups~\cite{Mya24}.

Isotypicity is a very strong relation on groups, which quite often implies
their isomorphism. This explains the need of the following definition. 

\begin{definition}
We say
that a group $G$ is \emph{defined by its types} if every group isotypic to~$G$ is isomorphic
to~$G$. 
\end{definition}

It was noticed in~\cite{Myas} that every finitely generated group $G$ which
is defined by its types satisfies a (formally) stronger property. 
Namely, we say
that 

\begin{definition}
A finitely generated group $G$ is \emph{strongly defined by types} if for any isotypic
to~$G$ group~$H$ every elementary embedding $G \to H$ is an isomorphism.
\end{definition}

Miasnikov and Romanovsky in the paper~\cite{Myas}  proved that 

1) every virtually polycyclic group is strongly defined by its types;

2) every finitely generated metabelian group is strongly defined by
its types;

3) every finitely generated rigid group is strongly defined by its types.
In particular, every free solvable group of finite rank is strongly defined by its
types.

\smallskip

R.\,Sclinos (unpublished) proved that finitely generated homogeneous groups are defined by types. Moreover, finitely generated co-hopfian and finitely presented hopfian groups are defined by types. Nevertheless, the main problem in the area remains widely open:

\begin{problem}[\cite{Mya6}] 
Is it true that every finitely generated group is defined by types?
\end{problem}

In the recent paper~\cite{Gvozd_new} Gvozdevsky proved that any field of finite transcendence degree over a prime subfield is defined by types.
Also he gave  several interesting  examples of certain countable isotypic but not isomorphic
structures: totally ordered sets, rings, and groups.



For arbitrary Abelian $p$-groups  it is possible to introduce invariants (like it was done in~\cite{a1} for elementary equivalence), which completely define isotypicity of any Abelian $p$-groups (and therefore of all periodic Abelian groups). We will give these invariants later and after that will formulate the main theorem describing, when two Abelian periodic groups are isotypically equivalent.

Since elementary equivalence is necessary for isotypicity, we will start with the results of elementary equivalence of Abelian groups.

\section{Elementary equivalence of Abelian groups}\leavevmode

\begin{definition}
Two groups are called \emph{elementarily equivalent} if their first order theories coincide.
\end{definition}

Elementary equivalent Abelian groups were completely described in 1955 by Wanda Szmielew in~\cite{a1} (see also Eclof and Fisher,~\cite{a2}).

To formulate her theorem we need to introduce a set of special invariants of Abelian groups.

Let $A$ be an Abelian group, $p$ a prime number, $A[p]$ be the subgroup of~$A$, containing all elements of~$A$ of the orders $p$ or~$1$ (it is so-called $p$-\emph{socle} of the group~$A$),
$kA$ be the subgroup of~$A$, containing all elements of the form $ka$, $a\in A$.

\medskip

{\bf The first invariant} is
$$
D(p;A):= \lim\limits_{n\to \infty} \dim ((p^nA)[p])\text{ for every prime }p.
$$

Note that for every $n\in \mathbb N$ the subgroup $(p^nA)[p]$ consists of elements which are annulated being multiplied by~$p$, therefore it is a vector space over the field~$\mathbb Z_p$. So $\dim (p^nA)[p]$ is for every $n\in \mathbb N$ a well-defined cardinal number.

Since $p^{n+1}A\subset p^n A$, then $(p^{n+1}A)[p]\subset (p^n A)[p]$, therefore
$$
\dim (p^{n+1}A)[p]\leqslant \dim (p^nA)[p] \text{ for all } n\in \mathbb N.
$$
Consequently we have a non-increasing sequence of cardinal numbers, which has the smallest element. 

Thus for any Abelian group $A$ and any prime $p$ the invariant  $D(p; A)$ is well-defined.

\medskip

{\bf The second invariant} is
$$
Tf (p;A):= \lim\limits_{n\to \infty} \dim (p^nA/p^{n+1}A).
$$

This  invariant  is also well-defined for any prime~$p$ and any Abelian group~$A$.

\medskip

{\bf The third invariant } is
$$
U(p,n-1;A):= \dim ((p^{n-1}A)[p] / (p^nA)[p]),
$$
which is called the  \emph{Ulm} invariant, it defines (in some sense)  the number of copies of $\mathbb Z_{p^n}$ in~$A$.

\medskip

{\bf The last invariant} is $\Exp(A)$ which is an \emph{exponent} of~$A$ (the smallest natural number  $n$ such that $\forall a\in A\, na=0$).

\begin{theorem}[Szmielew theorem on elementary classification of Abelian groups,~\cite{a1}]\label{EE-Abelian}
Two Abelian groups  $A_1$ and $A_2$ are elementarily equivalent if and only if their elementary invariants  $D(p;\cdot)$, $Tf (p;\cdot)$, $U(p,n-1;\cdot)$ and $\Exp(\cdot)$
pairwise coincide for all natural $n$ and prime~$p$ \emph{(}more precisely, they are either finite and coincide or simultaneously are equal to infinity\emph{)}.
\end{theorem}

\section{Periodic Abelian groups and their structure}\leavevmode

In this section we will formulate the most important (and useful for our main results) definitions and theorems about periodic Abelian groups. We will mostly use the book~\cite{Fuks}.

\medskip

Any periodic Abelian group is a direct sum of its $p$-components for different prime~$p$:
$$
A=\bigoplus_{p\text{ is prime}}A_p.
$$

Therefore we are interested in the structure and properties of Abelian $p$-groups: groups, where all elements have order $p^n$, $n\in \mathbb N\cup \{ 0\}$.

\medskip 

Let $p$ be some prime number, $A_p$ be an Abelian $p$-group.

It is said that an~element~$a\in A$ is \emph{divisible by}
a~positive integer~$n$ (denoted as $n\mid a$)
if there is an element $x\in A$ such that $nx=a$. A~group~$D$ is called \emph{divisible} if $n\mid a$ for all $a\in D$ and all natural~$n$.
The groups $\mathbb Q$ and $\mathbb Z(p^\infty)$ are examples of divisible groups. Any divisible Abelian group is a direct sum of the groups $\mathbb Q$ and $\mathbb Z(p^\infty)$ for different prome~$p$.  Any divisible subgroup of an Abelian group is a direct summand of this  group. A~group~$A$ is called \emph{reduced} if it has no nonzero divisible subgroups.

For every $A_p$ we have 
$$
A_p=A_{p,d}\oplus A_{p,r}=\bigoplus_{\varkappa_p} \mathbb Z(p^\infty)\oplus A_{p,r},
$$
where $A_{p,d}$ is the divisible part of~$A_p$, and $A_{p,r}$ is a reduced $p$-group. Therefore for any periodic Abelian~$A$ we have  $A=A_d\oplus A_r$, where 
$$
A_d\cong \bigoplus_{p\text{ is prime}} \left( \bigoplus_{\varkappa_p} \mathbb Z(p^\infty)\right),
$$
$A_r$ is a reduced periodic group. 

\medskip

A~subgroup $G$ of a~group~$A$ is called \emph{pure} if the equation $nx=g\in G$ is
solvable in~$G$ whenever it is solvable in the entire group~$A$.
In other words, $G$~is pure if and only if
$$
\forall n\in \mathbb Z\quad nG=G\cap nA.
$$

A~subgroup~$B$ of a~group~$A$ is called
a~$p$-\emph{basic subgroup} if it satisfies the following conditions:
\begin{enumerate}
\item
$B$ is a~direct sum of cyclic $p$-groups and infinite
cyclic groups;
\item
$B$ is pure in~$A$;
\item
$A/B$ is $p$-divisible.
\end{enumerate}

Every group, for every prime~$p$, contains $p$-basic
subgroups~\cite{Fuks}.

For $p$-groups $p$-basic subgroups are particularly
important. If $A$~is a~$p$-group and $q$ is a~prime different from~$p$,
then evidently $A$ has only one $q$-basic subgroup, namely~$0$.
Therefore, in $p$-groups we may refer to the $p$-basic subgroups
simply as \emph{basic} subgroups, without confusion.

We need the following facts about basic subgroups.

\begin{theorem}[\cite{Sele7}]\label{BasicMaxBounded}
Assume that $B$ is a~subgroup of a~$p$-group~$A$,
$B=\bigoplus\limits_{n=1}^\infty B_n$, and $B_n$ is a~direct
sum of groups $\mathbb Z(p^n)$.
Then $B$~is a~basic subgroup of~$A$ if and only if
for every integer $n> 0$, the subgroup $B_1\oplus \dots\oplus B_n$
is a~maximal $p^n$-bounded direct summand of~$A$.
\end{theorem}

Any Abelian $p$-group~$A$ is a direct sum of its divisible part~$D$ (isomorphic to $\bigoplus\limits_{\varkappa_0} \mathbb Z(p^{\infty})$) and its reduced part~$\overline A$ with a basic subgroup
$$
B= \bigoplus\limits_{n=1}^\infty \left( \bigoplus\limits_{\varkappa_n} \mathbb Z(p^n)\right).
$$
The basic subgroup~$B$ is dense in~$\overline A$ in its $p$-adic topology.

\begin{definition}
A finite system $\{ a_1,\dots, a_n\}$  of elements of an Abelian group~$A$ is called \emph{independent}, if for any $m_1,\dots, m_n\in \mathbb Z$ 
$$
m_1 a_1 +\dots +m_n a_n=0 \Longleftrightarrow  m_1a_1=m_2a_2 = \dots = m_na_n=0.
$$

An infinite system $L=\{ a_i\}_{i\in I}$ of elements of the group~$A$ is called \emph{independent} if every finite subsystem of~$L$ is independent. 

An independent system~$M$ of~$A$ is \emph{maximal} if there is no independent system in~$A$ containing~$M$ as a proper subsystem. 
\end{definition}

\begin{definition}
Given  $a\in A$,
 the greatest nonnegative integer~$r$ for which the equation
 $p^rx=a$ is solvable for some $x\in A$, is called the \emph{$p$-height} $h_p(a)$ 
of~$a$.
If  $p^r x=a$ is solvable for all $r$ is,  $a$ is of
 \emph{infinite $p$-height},  $h_{p}(a)=\infty$. If it is completely clear
from the context which prime~$p$ is meant, we  call $h_p(a)$ simply 
the \emph{height} of~$a$ and write $h(a)$.
\end{definition}

\begin{definition}
A reduced Abelian $p$-group $A$ is called \emph{separable}, if it does not contain any non-zero elements of infinite height.
\end{definition}

For a reduced $p$-group $A$ \emph{the first Ulm subgroup} of~$A$ is
$$
A^{\mathbf 1}:=\bigcap_{n=1}^\infty p^n A,
$$
it is the subgroup of~$A$ consisting of all elements of~$A$ of infinite height. Therefore a reduced $A$ is separable if and only if $A^{\mathbf 1}=0$.


\begin{proposition}[see \cite{Fuks}]
For a $p$-group $A$ and elements  $a_1,\dots, a_n\in A$ they belong to one finite
direct summand of~$A$ if and only if all nonzero
$m_1 a_1 + \dots + m_n a_n$ do not have  infinite height.
\end{proposition}

\begin{theorem}[see \cite{Fuks}, \cite{Bojer1}]\label{Th_Boyer}
Suppose that $B$ is a subgroup of $p$-group~$A$,
$$
B=B_1\oplus B_2\oplus\dots\oplus B_n\oplus \dots,
$$
where 
$$
B_n\cong \bigoplus_{\mu_n} \mathbb Z(p^n).
$$
The subgroup $B$ is a basic subgroup of~$A$ if and only if
$$
A=B_1\oplus B_2\oplus \dots\oplus B_n\oplus (B_n^*+p^n A),
$$
where $n\in \mathbb N$,
$$
B_n^*=B_{n+1}\oplus B_{n+2}\oplus \dots
$$
\end{theorem}

\medskip





\begin{definition}
	By the \emph{rank} $r(A)$ of a~group~$A$ we mean the cardinality of a~maximal independent system containing only elements of infinite and prime power orders.
	
	The \emph{final rank} $\mathrm{fin}\, r(A)$  of a $p$-group~$A$ is the minumum of all cardinal numbers $\mathrm{rank}\, (p^n A)$, $n=1,2,\dots$ (see~Sele, \cite{Sele7}). 
\end{definition}

\begin{example}
	If 
	$$
	A=B=B_1\oplus B_2\oplus \dots \oplus B_n \oplus \dots,\qquad B_n=\bigoplus_{\mu_n} \mathbb Z(p^n),
	$$
	then 
	$$
	\mathrm{rank}\, (p^n A)= \mathrm{rank}\, (B_{n+1}\oplus \dots)=\sum_{i=n+1}^\infty \mu_i
	$$
	and therefore $\mathrm{fin}\, r(A)=0$ if and only if $\Exp A =0$, $\mathrm{fin}\, r(A)=\infty$ if and only if $\Exp A=\infty$ (see~\cite{Fuks}, \S\,35).
\end{example}

\medskip

Let us consider now reduced Abelian $p$-groups.

If $A^{\mathbf 1}$ is the (first) Ulm subgroup of~$A$, then for any ordinal number~$\alpha$
we define
$$
A^{\alpha+1}:= (A^\alpha)^{\mathbf 1}\text{ and } A^\alpha:= \bigcap_{\gamma < \alpha} A^\gamma,
$$
if $\alpha$ is a limit ordinal number.

Then it is defined the well-ordered sequence of subgroups
$$
A=A^0\supset A^{\mathbf 1} \supset \dots \supset A^\alpha \supset \dots \supset A^\varkappa=0
$$
for some ordinal number~$\varkappa$.

The subgroup $A^\alpha$ is called the \emph{$\alpha$-th Ulm subgroup} of~$A$, and the quotient group $A_\alpha=A^\alpha / A^{\alpha+1}$ is called the \emph{$\alpha$-th Ulm factor} of~$A$.

The well-ordered sequence 
$$
A_0, A_1,\dots , A_\alpha, \dots \qquad (\alpha < \varkappa)
$$
is called the \emph{Ulm sequence} of~$A$, and $\varkappa$ is the \emph{Ulm type} of~$A$.

All Ulm factors of~$A$ are separable $p$-groups and they all  (except maybe the last one) are unbounded.

The following theorem is very important for us, but we will use only the first Ulm subgroup.

\begin{theorem}[\cite{Fuks}, Theorem 76.1]
	Let $\mu$ be a cardinal number,  $\varkappa$ be an ordinal number, 
	$$
	A_0, A_1,\dots , A_\alpha, \dots \qquad (\alpha < \varkappa)\eqno (1)
	$$
	be a sequence of nonzero $p$-groups. 
	
	A reduced $p$-group~$A$ of  the cardinality~$\mu$ and Ulm lenght~$\varkappa$ with the Ulm sequence~\emph{(1)}, exists if and only if the following holds:
	
	\emph{(a)} all $A_\alpha$, $\alpha < \varkappa$, are separable;
	
	\emph{(b)} $\sum\limits_{0\leqslant \alpha < \varkappa} |A_\alpha | \leqslant \mu \leqslant \prod\limits_{0\leqslant \mu < \min (\omega, \varkappa)} |A_n|$;
	
	\emph{(c)} $r (B_{\alpha+1}) \leqslant \mathrm{fin}\, r(A_{\sigma})$ for all $\alpha+1< \varkappa$, where $B_\alpha$ is a basic subgroup of~$A_\alpha$;
	
	\emph{(d)} $\sum\limits_{\lambda < \alpha < \varkappa} |A_{\alpha}|\leqslant |A_{\lambda}|^{\aleph_0}$ for any $0\leqslant \lambda < \varkappa$.
\end{theorem}

We are mostly interested in Abelian groups of the Ulm length~$2$ (i.\,e., $A^{\mathbf 1}\ne 0$ and $A^{\mathbf 2}=0$). 

For such groups the conditions (a)--(d) can be rewritten as

(a) $A_0, A_1$ are separable;

(b) $|A_0|+|A_1| \leqslant \mu \leqslant |A_0| \times |A_1|$. Since $A_0$ is infinite, then $|A_0|+|A_1|=\max (|A_0|, |A_1|)$ and $|A_0|\cdot |A_1|=\max (|A_0|, |A_1|)$ and therefore this condition means $\mu = \max (|A_0|, |A_1|)$;

(c) $r(B_1) \leqslant \mathrm{fin}\, r(A_0)$

(d) $|A_1|\leqslant |A_0|^{\aleph_0}$.

This means that any separable countable or finite separable Abelian $p$-group is possible as the first Ulm subgroup~$A^{\mathbf 1}$, which is the most important corollary for our needs. 

\section{Elementary equivalence of periodic Abelian groups}\leavevmode

\subsection{Elementary equivalence of Abelian $p$-groups}\leavevmode
Let us first concentrate on elementary equivalence of two Abelian $p$-groups.

Suppose that $A=D\oplus \overline A$, where $\overline A$ is reduced and has $B$ as its basic subgroup.

Assume also that
$$
D\cong \bigoplus_{\varkappa_0} \mathbb Z(p^\infty),\quad B_k\cong \bigoplus_{\varkappa_k} \mathbb Z(p^k),\quad k=1,\dots, n,\dots
$$

{\bf 1.} For the first invariant $D(p;A)$ let us fix some $n\in \mathbb N$ and represent $A$ as
$$
A=D\oplus B_1\oplus B_2\oplus \dots \oplus B_n \oplus (B_n^*+p^n\overline A)
$$
as in Theorem~\ref{Th_Boyer}.

Then 
$$
p^n A=(p^nD)\oplus p^n(B_n^* +p^n\overline A)= D\oplus p^n(B_n^* +p^n\overline A),
$$
therefore 
$$
\dim p^nA [p]=\varkappa_0 + \dim p^n(B_n^* +p^n\overline A)[p].
$$
If $B$ is bounded (in this case always $A^{\mathbf 1}=0$), then 
$$
\lim\limits_{n\to \infty} \dim p^n(B_n^* +p^n\overline A)[p]=0
$$
and 
$$
D(p;A)=\varkappa_0.
$$ 
If $B$ is unbounded, then 
$$
\varkappa_0+\lim\limits_{n\to \infty} \dim p^n(B_n^* +p^n\overline A)[p]=\varkappa_0+\mathrm{fin}\, r(B)+ \dim A^{\mathbf 1}[p]=\max (\varkappa_0, \mathrm{fin}\, r(B), \dim A^{\mathbf 1}[p] )= \infty.
$$

This shows that for $p$-group we have two cases: If $B$ is bounded, we have $\varkappa_0$, if not, then we have $D(p;A)=\infty$.
\medskip

{\bf 2.} The second invariant is
\begin{multline*}
\lim\limits_{n\to \infty} \dim (p^nA/p^{n+1}A)=\lim\limits_{n\to \infty} \dim (p^nD/p^{n+1}D \oplus p^n\overline A/p^{n+1}\overline A)=\lim\limits_{n\to \infty} \dim (p^n\overline A/p^{n+1}\overline A)=\\
=\lim\limits_{n\to \infty} \dim  (p^n B_1/p^{n+1}B_1+\dots+ p^nB_{n+1}/p^{n+1}B_{n+1}+p^n(B_{n+1}^*+p^{n+1}\overline A)/p^{n+1}(B_{n+1}^*+p^{n+1}\overline A))=\\
=\lim\limits_{n\to \infty} (\varkappa_{n+1}+\varkappa_{n+2}+\dots)= \mathrm{fin}\, r(B).
\end{multline*}

\medskip

{\bf 3.} The third invariant is
$$
U(p,n-1;A):= \dim ((p^{n-1}A)[p] / (p^nA)[p])=\varkappa_n.
$$

\medskip

Therefore \emph{two Abelian $p$-groups are elementarily equivalent if and only if  their basic subgroups are elementarily equivalent and in the case of bounded basic subgroups they have the same ranks of their divisible parts. }

\subsection{Periodic Abelian groups}\leavevmode

Now let us assume that 
$$
A=\bigoplus_{p\text{ is prime}} A_p,
$$
where every $A_p$ is a $p$-group.

For every prime $p$ all invariants (except $\Exp A$) give the results $0$ for all $A_q$, $q\ne p$, since $p\cdot A_q=A_q$ and $A_q [p]=0$. 
Therefore if two periodic Abelian groups are elementarily equivalent, then all their $p$-components are elementarily equivalent. 

If for two Abelian periodic groups $A=  \bigoplus\limits_{p\text{ is prime}} A_p$ and $A'=  \bigoplus\limits_{p\text{ is prime}} A_p'$ the corresponding components are elementarily equivalent: $A_p\equiv A_p'$ for all prime~$p$, then all invariants $D(p;\dots)$, $Tf(p;\dots)$, $U(p,n; \dots)$ coincide automatically, and we only need to consider $Exp(\dots)$. 

If for at least one $p$ the subgroups $A_p$ and $A_p'$ are unbounded, then the groups $A$ and~$A'$ are also unbounded and $\Exp(A)=\Exp(A')=\infty$. 

If all $A_p$, $A_p'$ are bounded, then all these groups are direct sums of cyclic groups and in this case we know all $m\in \mathbb N$, for which $\mathbb Z(m)\subset A$ and $\mathbb Z(m)\subset A'$. Therefore $\Exp A=\Exp A'$.

So we see that two periodic Abelian groups are elementarily equivalent if and only if their $p$-components are elementarily equivalent for all prime~$p$.

\section{Abelian $p$-groups with separable reduced parts}\leavevmode

In this section we will consider a particular case: Abelian $p$-groups with separable reduced parts (with zero first Ulm subgroups).

\subsection{Types of elements}\leavevmode

Suppose that the decomposition $A=D\oplus \overline A$ is fixed. Suppose also that we have $m$-tuple $(g_1,\dots, g_m)$ of elements of~$A$ and its type $\mathrm{tp}(g_1,\dots, g_m)$. If $g_i=d_i+a_i$, $i=1,\dots, m$, is a decomposition of these elements with respect to the direct summands $D$ and $\overline A$, then $\mathrm{tp}(d_1,\dots,d_m,a_1,\dots a_m)$ contains  $\mathrm{tp}(g_1,\dots,g_m)=\mathrm{tp}(d_1+a_1,\dots,d_m+a_m)$. Therefore we always can assume that we study only types $\mathrm{tp}(d_1,\dots,d_\ell, a_1,\dots, a_m)$, $d_1,\dots,d_\ell\in D$, $a_1,\dots, a_m\in \overline A$.

\medskip

Let us for some fixed decomposition $A=D\oplus \overline A$  have $d_1,\dots,d_\ell\in D$ and $a_1,\dots, a_m\in \overline A$. The elements $d_1,\dots, d_\ell$ generate a direct summand 
$$
\langle d_1,\dots, d_\ell\rangle=D_1\cong \bigoplus\limits_{n_0} \mathbb Z(p^\infty)\text{ in }D.
$$
Knowing $\mathrm{tp}(d_1,\dots, d_\ell)$ we can easily define~$n_0$: it is $k-1$, where $k$ is the minimal natural number such that there exist a subset $\{ m_1,\dots, m_k\} \subset \{ 1,\dots ,\ell\}$, $\alpha_1,\dots, \alpha_k\in \mathbb Z$, $0\leqslant \alpha_i< \mathrm{ord} (d_{m_i})$, $\alpha_1,\dots, \alpha_k$ are not all zeros, such that
$$
\alpha_1 d_{m_1}+\dots + \alpha_k d_{m_k} =0.
$$

Now let us study the elements $a_1,\dots, a_m$. All these elements have finite heights. Suppose that
$$
\mathrm{ord}(a_i)=p^{t_i},\quad h(a_i)= p^{s_i}, \quad i=1,\dots,\ell.
$$
Let us take instead of every~$a_i$ an element~$b_i$ such that $p^{s_i}b_i=a_i$. Then every $b_i$ has the height~$0$ and therefore generates a direct summand 
$$
a_i\in \langle b_i\rangle \cong \mathbb Z(p^{t_i+s_i}).
$$

If we take the subgroup $\overline B=\langle b_1,\dots,b_m\rangle$, it is a finite subgroup of~$\overline A$. Since $\overline A$ does not contain any elements of infinite height and $\overline B$ is finite, then the heights of all nonzero elements of~$\overline B$ in~$\overline A$ are bounded, therefore $\overline B$ is embedded in a direct summand of~$\overline A$. This summand is finite and is isomorphic to
$$
\overline B_1 \oplus \dots \oplus \overline B_q,\text{ where }\overline B_i \cong \bigoplus\limits_{n_i}\mathbb Z(p^i).
$$
Of course all $n_i$ are defined by formulas with $a_1,\dots, a_m$ as parameters. 

Also, we can find a basic subgroup $B$ of~$\overline A$ such that $\overline B$ is a direct summand of~$B$.

\smallskip

\subsection{Isotypical equivalence of Abelian $p$-groups with separable reduced parts}

\begin{theorem}\label{TheorA}
	For any prime $p$ two Abelian $p$-groups $A_1$ and $A_2$  with separable reduced parts are isotypic if and only if their divisible parts and their basic subgroups  are elementarily equivalent.
\end{theorem}

\begin{proof}
	Let $A_1$ and $A_2$ be two  Abelian $p$-groups with elementarily equivalent divisible parts and elementarily equivalent basic subgroups, $A_1=D_1\oplus \overline A_1$, $A_2=D_2\oplus \overline A_2$ be their decompositions in the direct sum of divisible and reduced separable subgroups, $(d_1,\dots, d_\ell,a_1,\dots,a_m)$ be a tuple of element of~$A_1$, where $d_1,\dots, d_\ell\in D_1$ and $a_1,\dots, a_m\in \overline A_1$. These elements generate (in the sense above) a direct summand of~$A_1$, isomorphic to
	$$
	C=C_d\oplus C_r\cong \bigoplus\limits_{n_0} \mathbb Z(p^\infty) \oplus \bigoplus\limits_{t=1}^m \left( \bigoplus\limits_{n_t} \mathbb Z(p^t)\right).
	$$
	Therefore if 
	$$
	D_1= \bigoplus\limits_{\varkappa_0} \mathbb Z(p^\infty), \quad D_2= \bigoplus\limits_{\varkappa_0'} \mathbb Z(p^\infty)
	$$
	and 
	$$
	B_1=\bigoplus\limits_{t=1}^\infty \left( \bigoplus\limits_{\varkappa_t} \mathbb Z(p^t)\right),\quad B_2=\bigoplus\limits_{t=1}^\infty \left( \bigoplus\limits_{\varkappa_t'} \mathbb Z(p^t)\right),
	$$
	then for all $i=0,1,\dots, m$
	$$
	n_i \leqslant \varkappa_i=\varkappa_i' \text{ or } \varkappa_i,\varkappa_i'\text{ are both infinite}.
	$$
	So we can denote
	$$
	A_1=D_1\oplus \overline A_1=(C_d\oplus \widetilde C_d)\oplus (C_r\oplus \widetilde B_1^{(1)}\oplus \dots
	\oplus \widetilde B_1^{(m)})\oplus (\overline B_1^{(m+1)}+p^m\overline A_1)=C\oplus \widetilde C_1.
	$$
	Similarly the group $A_2$ can be decomposed as
	$$
	A_2=D_2\oplus \overline A_2 \cong(C_d\oplus \widetilde C_d')\oplus (C_r\oplus \widetilde B_2^{(1)}\oplus \dots
	\oplus \widetilde B_2^{(m)})\oplus (\overline B_2^{(m+1)}+p^m\overline A_2)=C\oplus \widetilde C_2
	$$
	and
	$$
	\widetilde C_1\equiv \widetilde C_2.
	$$
	Therefore we have $(d_1',\dots, d_\ell',a_1',\dots, a_m')\in C\in A_2$ (the same as $(d_1,\dots, d_\ell,a_1,\dots, a_m)\in C\in A_1$ such that 
	$$
	\mathrm{tp}(d_1',\dots, d_\ell',a_1',\dots, a_m')=\mathrm{tp}(d_1,\dots, d_\ell,a_1,\dots, a_m).
	$$
	
	Therefore $A_1$ and $A_2$ are isotypic.
\end{proof}

\begin{remark}
	From~\cite{Fuks}, Theorem 77.3, it follows that any countable reduced separable Abelian $p$-group is a direct sum of cyclic groups. Therefore for countable groups with separable reduced parts  isotypicity coincides with isomorphism.
	
	For non-countable Abelian $p$-groups with separable reduced parts there are many examples of  isotypical and non-isomorphic groups. 
\end{remark}

\subsection{Strong homogeneity}

\begin{definition}
	A model $\mathcal M$ is called \emph{$\omega$-strongly homogeneous}, if for any $a_1,\dots, a_n, b_1,\dots, b_n\in M$ if $\mathrm{tp}\, (a_1,\dots, a_n)=\mathrm{tp}\, (b_1,\dots, b_n)$, then there exists an automorphism $\varphi \in \Aut \mathcal M$ such that $\varphi (a_i)=b_i$, $i=1,\dots, n$.
\end{definition}

From the proof of Theorem~\ref{TheorA} we can derive the following corollary:

\begin{corollary}
	For any prime $p$ any Abelian $p$-group $A$  with a separable reduced part is $\omega$-strongly homogeneous.
\end{corollary}

\begin{proof}
	Let us take our group~$A$ and $a_1,\dots, a_n, b_1,\dots, b_n\in A$ with $\mathrm{tp}\, (a_1,\dots, a_n)=\mathrm{tp}\, (b_1,\dots, b_n)$. According to the previous considerations we can (without loss of generality) assume, that $a_1,\dots, a_k, b_1,\dots, b_k \in D$, where $D$ is the divisible part of~$A$, and $a_{k+1},\dots, a_n \in A_1$, $b_{k+1},\dots, b_n\in A_2$, where $A_1$ and $A_2$ are reduced and $A=D\oplus A_1=D\oplus A_2$.
	
	Since $\mathrm{tp}\, (a_1,\dots, a_k)=\mathrm{tp}\, (b_1,\dots, b_k)$, these elements are contained in isomorphic minimal direct summands of~$D$: 
	$$
	D_1\cong D_2\cong \bigoplus\limits_{\ell} \mathbb Z(p^\infty).
	$$
	In these groups $a_1,\dots, a_k$ and $b_1,\dots, b_k$ are such elements that for any $\alpha_1,\dots, \alpha_k\in \mathbb Z$ the linear combinations $\alpha_1 a_1+\dots +\alpha_k a_k$ and $\alpha_1 b_1+\dots +\alpha_k b_k$ have the same orders. Therefore there exists an isomorphism $\varphi_1: D_1\to D_2$ such that $\varphi_1 (a_i)=b_i$ for all $i=1,\dots, k$. 
	
	Since $D=D_1\oplus D_1'=D_2\oplus D_2'$, where $D_1' \cong D_2'$, the isomorphism $\varphi$ can be extended up to an automorphism $\varphi_2\in \Aut D$.
	
	Now let us consider the elements $a_{k+1},\dots ,a_n\in A_1$ and $b_{k+1},\dots, b_n\in A_2$. Since  $\mathrm{tp}\, (a_{k+1},\dots, a_n)=\mathrm{tp}\, (b_{k+1},\dots, b_n)$, these elements are contained in isomorphic minimal direct summands $B_1$ of~$A_1$ and $B_2$~$A_2$, respectively. 
	
	By the same reasons, since for any $\alpha_{k+1},\dots, \alpha_n\in \mathbb Z$ the  linear combinations $\alpha_{k+1} a_{k+1}+\dots +\alpha_n a_n$ and $\alpha_1 b_1+\dots +\alpha_k b_k$ have the same orders and heights, there exists an isomorphism $\varphi_3: B_1\to B_2$ such that $\varphi_3 (a_i)=b_i$ for all $i=k+1,\dots, n.$ 
	
	Since $A_1=B_1\oplus C_1$ and $A_2=B_2\oplus C_2$, where $C_1\cong C_2$, we can extend $\varphi_3$ up to an isomorphism $\varphi_4: A_1\to A_2$. 
	
	Since $A=D\oplus A_1=D\oplus A_2$, the mappings $\varphi_2$ and $\varphi_4$ give us the required automorphism~$\varphi$.
\end{proof}

In next sections we will consider the general case of periodic Abelian groups.

\section{Types of elements in periodic Abelian groups}\leavevmode

\subsection{Formulas in Abelian groups}\leavevmode

The most important for us are the following two statements about elimination of quantifiers in Abelian groups:

\begin{lemma}[\cite{Mya2}, Lemma A.2.1, Pr\"ufer]\label{LemmaA21}
	Every positive primitive  formula in the language of Abelian groups  is equivalent modulo $T_{ab}$ (theory of Abelian groups) to a finite conjunction of formulas of the forms $\alpha_1 x_1+\dots +\alpha_k x_k=0$ and and  $p^n |  \alpha_1 x_1+\dots +\alpha_k x_k$, where $\alpha_1,\dots, \alpha_k\in \mathbb Z$,  $p$ is primes and $n$ is a natural number. 
\end{lemma}

\begin{proposition}[\cite{Baur}, Baur]
Every formula in the language of Abelian groups is equivalent relative to the theory of abelian groups to a boolean combination of $\forall \exists$-sentences and positive primitive formulas from Lemma~\ref{LemmaA21}.
\end{proposition}

According to these two statements we have the following situation: 
the type of any tuple $(a_1,\dots, a_m)\in A$ (here $A$ is any Abelian group, not necessarily periodic) contains two parts: 

(1) Elementary theory of~$A$;

(2) All formulas
$$
\alpha_1 a_1+\dots +\alpha_m a_m=0,\quad p^n | \alpha_1 a_1+\dots +\alpha_m a_m,
$$
with $\alpha_1,\dots, \alpha_m\in \mathbb Z$, $p$ prime and $n$ natural, that hold in~$A$.

For a $p$-group $A$ it could be written shorter:

(1) Elemenary theory of~$A$;

(2) All formulas
$$
o(\alpha_1 a_1+\dots +\alpha_m a_m)=o_{\alpha_1,\dots, \alpha_m}=p^k,\quad h( \alpha_1 a_1+\dots +\alpha_m a_m)=h_{\alpha_1,\dots, \alpha_m}=p^n,
$$
with different $\alpha_1,\dots, \alpha_m\in \mathbb Z$.

In other words, for any Abelian $p$-group $A$ possible types of elements describe the elementary theory of~$A$, and orders and heights of linear combinations of the corresponding elements.
All orders are different powers of~$p$, all heights are also different powers of~$p$, or $\infty$.

\subsection{Reduction of types over periodic Abelian groups to $p$-groups}\leavevmode

Assume that $A$ is a periodic Abelian group and $A=\bigoplus\limits_{p\text{ is prime}} A_p$, where  every $A_{p}$ is a $p$-group.

If $\mathrm{tp}\, (A)= \mathrm{tp}\, (A')$, then clearly for all prime~$p$ we have $\mathrm{tp}\, (A_p) = \mathrm{tp}\, (A_p')$, because:

(1) If $\mathrm{tp}\, (A)= \mathrm{tp}\, (A')$, then $A\equiv A'$, then $A_p\equiv A_p'$, so these $p$-components $A_p$ and $A_p'$ are elementarily equivalent.

(2) If $o( \alpha_1 g_1+\dots +\alpha_m g_m)=p^k$ for $g_1,\dots, g_m\in A_p$ in the whole group~$A$, then $o( \alpha_1 g_1+\dots +\alpha_m g_m)=p^k$ in~$A_p$.

(3) If $h( \alpha_1 g_1+\dots +\alpha_m g_m)=p^k$ or~$\infty$ for $g_1,\dots, g_m\in A_p$ in the whole group~$A$, then $h( \alpha_1 g_1+\dots +\alpha_m g_m)=p^k$ or $\infty$ in~$A_p$ (since $A_p$ is pure in~$A$).

\medskip

Conversly, if for all prime~$p$ we have $\mathrm{tp}\, (A_p) = \mathrm{tp}\, (A_p')$, then  $\mathrm{tp}\, (A)=\mathrm{tp}\, (A')$, because:

(1) If $\mathrm{tp}\, (A_p)= \mathrm{tp}\, (A_p') $ for all prime $p$, then $A_p\equiv A_p'$ for all prime~$p$, then $A\equiv A'$, so the whole groups $A$ and $A'$ are elementarily eqivalent.

(2) If $a\in A$ and $a=a_1+\dots +a_m$, where $a_i \in A_{p_i}$, all $p_1,\dots, p_m$ are distinct, then $o(a)=o(a_1)\dots o(a_m)$.

(3) If $a\in A$ and $a=a_1+\dots +a_m$, where $a_i \in A_{p_i}$, all $p_1,\dots, p_m$ are distinct, $h(a_i)=p_i^{k_i}$, then $p_i$-height of~$a$  is precisely~$p_i^{k_i}$ for all $p_i$, and $q$-height of~$a$ is infinite for all $q\ne p_1,\dots, p_m$.

So we see that two Abelian periodic groups are  equivalent by types if and only if their $p$-components are equivalent by types for all prime~$p$.

Therefore from this moment we will assume that our groups $A$ and $A'$ are $p$-groups.

\subsection{Sums of divisible and bounded $p$-groups}\leavevmode

If
 $A=D\oplus B$, where $D$ is divisible and $B=\bigoplus\limits_{n=1}^N B_n$ is  bounded, then elementary theory of~$A$ gives us $\mathrm{rank}\, D=\varkappa_0$ and $\mathrm{rank}\, B_n=\varkappa_n$ for all $n=1,\dots, N$.
 
This case is simple: every element $g\in A$ is a sum of its projections to $D, B_1,\dots, B_N$ and types of elements of these groups are absolutely defined by their elementary types. 

Therefore in this case two $p$-groups are isotypically equivalent if and only if they are elementary equivalent.

\subsection{Abelian $p$-groups with unbounded basic subgroups}\leavevmode

We remember that if $A=D\oplus \overline A$ is a decomposition of~$A$ into divisible and reduced parts and $B$ is a basic subgroup of~$\overline A$, then if $B$ is unbounded, then elementary theory of~$A$ gives us only the ranks $\varkappa_n$ of~$B_n$, $n=1,2,\dots$.

Consider the first Ulm subgroup $A^{\mathbf 1}$ of~$\overline A$ and assume that its basic subgroup is
$$
C=\bigoplus_{k=1}^\infty C_k,\quad C_k\cong \bigoplus_{\gamma_k} \mathbb Z(p^k).
$$

Now we can formulate the main classification theorem for isotypical equivalence of Abelian $p$-groups.

\begin{theorem}
	If two Abelian $p$-groups $A$ and $A'$ have $\varkappa_0$ and $\varkappa_0'$ as ranks of their divisible parts, $\varkappa_i, \varkappa_i'$, $i=1,\dots, n,\dots$ as ranks of the corresponding direct summands $B_i,B_i'$ of their basic subgroups $B,B'$, and $\gamma_i, \gamma_i'$, $i=1,\dots, n,\dots$, as ranks of the corresponding direct summands $C_i,C_i'$ of the basic subgroups $C,C'$ of their first Ulm subgroups, then 
	$$
	\mathrm{tp}(A) = \mathrm{tp} (A')
	$$
	if and only if 
	$$
	\varkappa_i=\varkappa_i'\text{ for  }i=1,2,\dots, n,\dots,
	$$
	and either
	$$
		\varkappa_0+\mathrm{fin}\, r(A^{\mathbf 1}), 	\varkappa_0'+\mathrm{fin}\, r({A'}^{\mathbf 1})=\infty,
$$
	or   
	\begin{align*}
		&	\varkappa_0= \varkappa_0'\text{ are finite},\\
		&\mathrm{fin}\, r(A^{\mathbf 1})= \mathrm{fin}\, r({A'}^{\mathbf 1})=0,\\
		&	\sum_{k=i}^N \gamma_k = \sum_{k=i}^N \gamma_k',\quad i=1,\dots, N,\text{ where } \gamma_k=\gamma_k'= 0 \text{ for all } k>N.
	\end{align*}
	Here by $\mu = \nu$ we mean that either these cardinals are finite and coincide, or they are both infinite.
\end{theorem}

\begin{proof}
{\bf Step 1.} \emph{If two Abelian $p$-groups are isotypically equivalent, then their invariant coincide}.

	Assume first that $A$ and $A'$ are isotypically equivalent. Since they are elementarily equivalent, we have $\varkappa_i=\varkappa_i'$ for all $i\in \mathbb N$. 
	
	Now let us take $N$ indepent elements $a_1,\dots, a_n\in A$ of infinite heights and of orders~$p^n$. Elements of orders~$p^n$ are independent, if all their non-zero linear combinations, where not all coefficients are divided by~$p$, have the order~$p^n$. Therefore the full type of these $a_1,\dots, a_N$ consists of three things:
	
	(1) Elementary theory of~$A$;
	
	(2) Orders of every $a_i$ and their independency;
	
	(3) Heights of all $a_1,\dots, a_N$ and their linear combinations are infinite.
	
	If the divisible part of~$A$ has an infinite rank, then we can find such $a_1,\dots, a_N$ for all natural $n$ and~$N$: it is enough to take elements $a_i$ of the order~$p^n$ from different direct summands of~$D$.  If the final rank of~$A^{\mathbf 1}$ is infinite, we also can find such $a_1,\dots, a_N$ for all natural $n$ and~$N$: it is enough to take $a_i=p^{m_i-n}b_i$, where $b_i$ are generators of different direct summands of~$C$ of the orders~$m_i\geqslant n$.
	
	Therefore if $\varkappa_0+\mathrm{fin}\, r(A^{\mathbf 1})=\infty$, then for any $n, N \in \mathbb N$ we can find in~$A$ (and therefore in~$A'$) $N$ independent $a_1,\dots, a_N$ of orders~$p^n$ and infinite heights. If $A'$ has a divisible part of finite rank and bounded basic subgroup of~${A'}^{\mathbf 1}$, then there exists some $N_0\in \mathbb N$ and some $n_0\in \mathbb N$ such  that in~$A'$ it is impossible to find $N_0$ independent elements of infinite heights and orders~$p^{n_0}$. 
	
	Consequently, if $\varkappa_0+\mathrm{fin}\, r(A^{\mathbf 1})=\infty$ and $A\equiv_{tp} A'$, then $\varkappa_0'+\mathrm{fin}\, r({A'}^{\mathbf 1})=\infty$.
	
	\medskip
	
	Now let us assume that $\varkappa_0+\mathrm{fin}\, r(A^{\mathbf 1})< \infty$ and $A\equiv_{tp} A'$, then $\varkappa_0'+\mathrm{fin}\, r({A'}^{\mathbf 1})<\infty$. It means that $\varkappa_0,\varkappa_0'< \infty$ and $A^{\mathbf 1}, {A'}^{\mathbf 1}$ are bounded, i.\,e., these groups are direct sums of cyclic groups:
	$$
	A^{\mathbf 1}=\bigoplus_{i=1}^N C_i,\qquad {A'}^{\mathbf 1}=\bigoplus_{i=1}^N C_i'.
	$$
	In this case in~$A$ there exist maximum $\varkappa_0$ independent elements $a_1,\dots, a_{\varkappa_0}$ of orders~$p^{N+1}$ and infinite heights, and in~$A'$ there exist maximum $\varkappa_0'$ independent elements $a_1',\dots, a_{\varkappa_0}'$ of orders~$p^{N+1}$ and infinite heights. It means that $\varkappa_0=\varkappa_0'$.
	
	If we take the maximal set of independent elements of the order $p^N$ and infinite height, this set in~$A$ consists of $\varkappa_0+\gamma_N$ elements and in~$A'$ of $\varkappa_0'+\gamma_N'$ elements. Therefore $\varkappa_0+\gamma_N=\varkappa_0'+\gamma_N'$. 
	
	Similarly the maximal number of  elements of the order $p^i$, $i< N$, in~$A$ consists of $\varkappa_0+\gamma_i+\gamma_{i+1}+\dots +\gamma_N$ elements and in~$A'$ of $\varkappa_0'+\gamma_i'+\gamma_{i+1}'+\dots+\gamma_N'$ elements. Since $\varkappa_0=\varkappa_0'$ are the same finite numbers, we have 
	$$
\sum_{k=i}^N \gamma_k = \sum_{k=i}^N \gamma_k',\quad i=1,\dots, N.
$$

We proved that if two Abelian $p$-groups are isotypically equivalent, then the corresponding invariants coincide.

\medskip

Now let us assume that two Abelian $p$-groups $A$ and $A'$ with unbounded basic subgroups have the same invariants from the theorem. 

\bigskip

{\bf Step 2.}  \emph{Extension of tuples of elements}.

\begin{lemma}
	If for two Abelian groups $A$ and $A'$ for any $a_1,\dots, a_n\in A$ there exist $b_1,\dots, b_N\in A$ such that $a_1,\dots, a_n$ are linear combinations of $b_1,\dots, b_N$ and there exist $c_1,\dots, c_N\in A'$ such that 
	$$
	tp^A (b_1,\dots, b_N)=tp^{A'} (c_1,\dots, c_N)
	$$
	and the same for any $a_1',\dots, a_n'\in A'$, then $A$ and $A'$ are isotypically equivalent.
\end{lemma}

\begin{proof}
	Clear.
\end{proof}

According to this lemma we can use instead of the type of any tuple $a_1,\dots, a_n$ the type of a tuple $b_1,\dots, b_N$, for which all $a_i$ are linear combinations of $b_1,\dots, b_N$.

\medskip

{\bf Step 3.} \emph{Types of extended tuples.}

\begin{lemma}\label{lemma_type}
	Let $a_1,\dots, a_n$ be a tuple of elements of an Abelian $p$-group~$A$. Then there exists a tuple $b_1,\dots, b_N$ of~$A$ such that $a_1,\dots, a_n\in \langle b_1,\dots, b_N\rangle$ and 
	
	\medskip
	
	\emph{(1)} $b_1,\dots, b_N$ are independent elements of the orders $p^{\ell_1},\dots, p^{\ell_N}$;
	
	\smallskip
	
	\emph{(2)} $b_1,\dots, b_r$ have $p$-heghts~$0$ and all their non-trivial linear cobinations have finite heights (and therefore they can be embedded into a finite direct summand of~$A$);
	
	\smallskip
	
	\emph{(3)} $b_{r+1},\dots , b_{r+s}$ have $p$-heights~$0$ and $p^{m_1},\dots, p^{m_s}$, where $m_1< \ell_{r+1}$, \dots  , $m_s< \ell_{r+s}$, are minimal natural numbers such that the elements $p^{m_1} b_{r+1},\dots, p^{m_s}b_{r+s}$ has infinite heights;
	
	\smallskip
	
	\emph{(4)} elements $p^{m_1} b_{r+1},\dots, p^{m_s}b_{r+s}, b_{r+s+1},\dots, b_N$ are independent elements of infinite $p$-heights.
\end{lemma}

\begin{proof}
Let us take all elements from the set $a_1,\dots, a_n$ of finite heights and find the corresponding elements $c_i$ of heights~$0$ such that $a_i=p^k c_i$. 	

After this action we can assume that all initial $a_1,\dots, a_n$ have either $p$-height~$0$, or infinite $p$-height.

Let us generate the subgroup $F=\langle a_1,\dots, a_n\rangle \subset A$, which is of course a finite Abelian group. Let us denote by $F_\infty$ the subgroup of~$F$ consisting of all elements of infinite heights. 

Having a finite Abelian group $F$ and some its subgroup $F_\infty$ we know that this group is a direct sum of cyclic groups $\mathbb Z(p^k)$ (for different natural~$k$), where some generators coincide with generators of~$F_\infty$ and some generators of~$F_\infty$ are generators of~$F$ multiplied by powers of~$p$.

Let us take this set of generators. In the case, if some generator~$c$ of the group~$F$ has a height $>0$ in the whole group~$A$, we can take the corresponding $b$ of height~$0$ such that $c=p^k\cdot b$, and use this $b$ instead of~$c$.

Finally we will come to the set $b_1,\dots, b_N$ with the properties stated in the lemma. 

\end{proof}
\medskip

{\bf Step 4.} \emph{If two Abelian $p$-groups have same invariants, then they are isotypically equivalent}.

\medskip

If we want to prove that $A$ and $A'$ are isotypically equivalent, then we have to consider any tuple $a_1,\dots, a_n\in A$ and find a tuple $a_1',\dots, a_n'\in A'$ such that $tp^A (a_1,\dots, a_n)=tp^{A'}(a_1',\dots, a_n')$.

In reality we can take the corresponding tuple  $b_1,\dots, b_N\in A$ from Lemma~\ref{lemma_type} and find a tuple $b_1',\dots, b_N'\in A'$ such that $tp^A (b_1,\dots, b_N) = tp^{A'} (b_1',\dots, b_N')$.

Since we know that in the group~$A$ the elements  $p^{m_1} b_{r+1},\dots, p^{m_s}b_{r+s}, b_{r+s+1},\dots, b_N$ are independent elements of infinite $p$-heights, then we know that in the Ulm subgroup $D\oplus A^{\mathbf 1}$ there exist independent direct summands of the orders 
$$
p^{\ell_{r+1}-m_1},\dots, p^{\ell_{r+s}- m_s}, p^{\ell_{r+s+1}},\dots, p^{\ell_N}.
$$
This knowelge gives us some conditions to the cardinalities $\gamma_1,\dots, \gamma_i,\dots$: if the number of $p^k$ in this sequence is $\mu_k$, $k=1,\dots, M$,  then  
\begin{align*}
	\sum_{i=1}^\infty \gamma_i &\geqslant \mu_1+\dots + \mu_M,\\
	\sum_{i=2}^\infty \gamma_i &\geqslant \mu_2+\dots+\mu_M,\\
	\dots\dots&\dots\dots\dots \dots \dots  \dots\\
	\sum_{i=M-1}^\infty \gamma_i&\geqslant \mu_{M-1}+\mu_M,\\
	\sum_{i=M}^\infty \gamma_i&\geqslant \mu_M,
\end{align*}
and therefore
\begin{align*}
	\sum_{i=1}^\infty \gamma'_i &\geqslant \mu_1+\dots + \mu_M,\\
	\sum_{i=2}^\infty \gamma'_i &\geqslant \mu_2+\dots+\mu_M,\\
	\dots\dots&\dots\dots\dots \dots  \dots \dots\\
	\sum_{i=M-1}^\infty \gamma'_i&\geqslant \mu_{M-1}+\mu_M,\\
	\sum_{i=M}^\infty \gamma'_i&\geqslant \mu_M.
\end{align*}

It means that in the group $D'\oplus {A'}^{\mathbf 1}$ there also exist independent direct summands of the orders
$$
p^{\ell_{r+1}-m_1},\dots, p^{\ell_{r+s}- m_s}, p^{\ell_{r+s+1}},\dots, p^{\ell_N}
$$
and we can choose them as potential $p^{m_1} b_{r+1}'=b_{r+1}'',\dots, p^{m_s}b_{r+s}'=b_{r+s}''$ and final  $b_{r+s+1}',\dots, b_N'$.

Now let us use the fact that a basic subgroup $B$ of $A$ is isomorphic to a basic subgroup of $A/(D\oplus A^{\mathbf 1})$. It is clear that the elements 
$$
b_1,\dots, b_r, b_{r+1},\dots, b_{r+s}
$$
are generators of independent cyclic summands of the basic subgroups of  $A/(D\oplus A^{\mathbf 1})$ of the orders $p^{\ell_1},\dots, p^{\ell_r},p^{m_1},\dots, p^{m_s}$, respectively.

As above, this knowelge gives us the following conditions to the cardinalities $\varkappa_1,\dots, \varkappa_i,\dots$: if the number of $p^k$ in the sequence $p^{\ell_1},\dots, p^{\ell_r},p^{m_1},\dots, p^{m_s}$ if $\nu_k$, $k=1,\dots, K$,  then  
$$
\varkappa_1\geqslant \nu_1,\ \varkappa_2\geqslant \nu_2,\dots, \varkappa_K\geqslant \nu_K,
$$
so
$$
\varkappa'_1\geqslant \nu_1,\ \varkappa'_2\geqslant \nu_2,\dots, \varkappa'_K\geqslant \nu_K
$$
and there exist independent $b_1',\dots, b_r', c_1',\dots, c_s'$ in a basic subgroup of~$A'$ of the zero heights and of the orders $p^{\ell_1},\dots, p^{\ell_r},p^{m_1},\dots, p^{m_s}$, respectively.

Let us take $c_1,\dots, c_s$ of nonzero heights such that 
$$
p^{m_1} c_1=b_{r+1}'',\ \dots,\ p^{m_s}c_s=b_{r+s}''
$$
(it is possible, because $b_{r+1}'',\dots, b_{r+s}''$ have infinite heights), and then 
$$
b_{r+1}': = c_1+c_1',\dots, b_{r+s}':= c_s+c_s'.
$$
Every $c_i+c_i'$ has a zero height as a sum of elements of zero and nonzero heights. Also $p^{m_i}(c_i+c_i')=p^{m_1} c_i + p^{m_1}c_i'=b_{r+i}'' + 0=b_{r+i}''$. 

Therefore now we have $b_1',\dots, b_N'\in A'$ with the same type as $b_1,\dots, b_N$, what was required.

Consequently our theorem is proved

\end{proof}

{\bf Acknowledgements.}
Our sincere thanks go to Alexey Miasnikov and Eugene Plotkin for very useful discussions regarding  this work and permanent attention to it.

\end{document}